\documentclass[12pt,a4paper]{article}

\usepackage[a4paper,margin=1in]{geometry}
\usepackage[T1]{fontenc}
\usepackage{microtype}
\usepackage{indentfirst}
\usepackage{titlesec}

\usepackage{amsmath,amssymb,amsthm}
\usepackage{newtxtext,newtxmath}
\allowdisplaybreaks[3]

\usepackage{enumitem}
\usepackage[hidelinks]{hyperref}
\hypersetup{
	pdftitle={Nice Partitions and Freeness Obstructions for Deformations of Graphic Arrangements},
	pdfauthor={Weikang Liang and Suijie Wang},
	pdfsubject={Hyperplane arrangements associated with gain graphs},
	pdfkeywords={hyperplane arrangements, gain graphs, affinographic arrangements, nice partitions, supersolvable arrangements, free arrangements, chordal graphs}
}

\titleformat{\section}{\normalfont\Large\bfseries}{\thesection.}{0.6em}{}
\titleformat{\subsection}{\normalfont\large\bfseries}{\thesubsection.}{0.6em}{}
\titlespacing*{\section}{0pt}{2.4ex plus 0.7ex minus 0.2ex}{1.2ex plus 0.2ex}
\titlespacing*{\subsection}{0pt}{2.0ex plus 0.5ex minus 0.2ex}{0.9ex plus 0.2ex}
\setlist[itemize]{leftmargin=2em,itemsep=0.2ex,topsep=0.4ex}
\setlist[enumerate]{leftmargin=2em,itemsep=0.2ex,topsep=0.4ex}

\theoremstyle{plain}
\newtheorem{theorem}{Theorem}[section]
\newtheorem{lemma}[theorem]{Lemma}
\newtheorem{proposition}[theorem]{Proposition}
\newtheorem{corollary}[theorem]{Corollary}

\theoremstyle{definition}
\newtheorem{example}[theorem]{Example}

\theoremstyle{remark}
\newtheorem*{remark*}{Remark}

\DeclareMathOperator{\codim}{codim}

\newcommand{\A}{\mathcal A}
\newcommand{\B}{\mathcal B}
\newcommand{\Sgain}{\mathcal S}
\newcommand{\K}{\mathbb K}

\numberwithin{equation}{section}
\begin{document}

\begin{center}
	\setlength{\parskip}{0pt}
	{\Large\bfseries
		Nice Partitions, Supersolvability, and Freeness in Deformations of Graphic Arrangements\par}
	
	\vspace{7pt}
	
	Weikang Liang$^{1}$,
	Suijie Wang$^{2,*}$
	and Yue Zhou$^{3}$\par
	
	\vspace{8pt}
	
	$^{1,3}$ 
	College of Science\\
	National University of Defense Technology\\
	Changsha 410073, Hunan, P. R. China\par
	
	\vspace{12pt}
	
	$^{2}$ 	School of Mathematics\\
	Hunan University\\
	Changsha 410082, Hunan, P. R. China\par
	
	\vspace{15pt}
	
	\begin{tabular}{@{}r@{\ }l@{}}
		Emails: 
		$^{1}$\href{mailto:liangweikang08@163.com}{liangweikang08@163.com},
		$^{2}$\href{mailto:wangsuijie@hnu.edu.cn}{wangsuijie@hnu.edu.cn}
		$^{3}$\href{mailto:yue.zhou.ovgu@gmail.com}{yue.zhou.ovgu@gmail.com}
	\end{tabular}\\[2pt]
	$^{*}$ Corresponding author\par
\end{center}
\vspace{1.2em}

\begin{abstract}
We study affine deformations \(\A(G_{\Sgain})\) of graphic arrangements
and their cones. Here \(G=([n],E(G))\) is a simple graph with finite gain
sets \(\Sgain=(S_{ij})\), and \(\A(G_{\Sgain})\) consists of the
hyperplanes \(x_i-x_j=a\) with \(a\in S_{ij}\).
We call \(G_{\Sgain}\) blockwise admissible if every block has a vertex
ordering \(v_1,\ldots,v_m\) satisfying
\(S_{v_kv_j}-S_{v_kv_i}\subseteq S_{v_iv_j}\) for every \(k\) and all
distinct \(i,j>k\). Every such ordering is a perfect elimination ordering.
We prove, for arbitrary \(G\), that
\[
\begin{aligned}
G_{\Sgain}\text{ is blockwise admissible}
&\Longleftrightarrow
c\A(G_{\Sgain})\text{ is supersolvable}\\
&\Longleftrightarrow
\A(G_{\Sgain})\text{ admits a nice partition}.
\end{aligned}
\]
We give a direct arrangement-theoretic proof of these equivalences.
For a block, the equivalence between admissibility and supersolvability
is already contained in Zaslavsky's characterization of supersolvable
graphic-lift lattices. On each block, we further show that every nice
partition is induced by an admissible ordering, and that its induced
edge classes are stars with distinct centers. Under these equivalent
conditions, we construct a maximal modular chain through the
hyperplane at infinity.
We also establish a necessary condition for freeness:
\[
c\A(G_{\Sgain})\text{ is free}
\quad\Longrightarrow\quad
G\text{ is chordal}.
\]
\end{abstract}



\section{Introduction}
Let \(G=([n],E(G))\) be a simple graph and let \(\K\) be a field.
The graphic arrangement \(\A(G)\) in \(\K^n\) consists of the
hyperplanes \(x_i-x_j=0\), where \(ij\in E(G)\). Its characteristic
polynomial is the chromatic polynomial of \(G\). For an arrangement
\(\B\), write \(L(\B)\) for its intersection poset, ordered by reverse
inclusion. A central arrangement \(\B\) is supersolvable if \(L(\B)\)
contains a maximal chain of modular elements. Stanley's theorem
characterizes supersolvability of \(\A(G)\) by chordality of \(G\)
\cite{Stanley1972,Stanley2007,MuStanley2015}: every cycle of length at
least four has a chord, or equivalently, \(G\) has a perfect elimination
ordering. A nice partition is an independent partition satisfying a
singleton condition in every nontrivial localization. Every maximal
modular chain induces a nice partition \cite{Terao1992}. For graphic
arrangements, conversely, the existence of a nice partition forces
chordality, and every nice partition is induced by a maximal modular chain
\cite{LiangWang2026}.

This paper studies affine deformations of graphic arrangements in which
each edge is replaced by a finite parallel class of hyperplanes. We
characterize nice partitions and supersolvability in terms of the graph
and the gain sets, and prove a graph-theoretic obstruction to freeness.
Let \(\Sgain=(S_{ij})\) be a family of finite subsets of \(\K\),
indexed by ordered pairs of distinct vertices, with
\(S_{ji}=-S_{ij}\) and \(S_{ij}\ne\emptyset\) exactly when
\(ij\in E(G)\). Here \(-S=\{-a\mid a\in S\}\). Write \(G_{\Sgain}\)
for \(G\) together with these gain sets, and put
\[
H_{ij}^{a}=\{x\in\K^n\mid x_i-x_j=a\},
\qquad
\A(G_{\Sgain})=\{H_{ij}^{a}\mid a\in S_{ij},\ i<j\}.
\]
These are affinographic arrangements associated with additive gain graphs
\cite{ForgeZaslavsky2007}; their intersection semilattices can also be
described using semimatroids \cite{Ardila2007}. Taking \(S_{ij}=\{0\}\)
on every edge recovers \(\A(G)\). Over a field of characteristic zero,
suitable gain sets on complete graphs give the Shi, Catalan, and Linial
arrangements \cite{PostnikovStanley2000}.

Introduce an additional coordinate \(y\), and let
\(cH_{ij}^{a}=\{(x,y)\in\K^{n+1}\mid x_i-x_j=ay\}\) and
\(H_\infty=\{y=0\}\). The cone of the affine arrangement is
\[
c\A(G_{\Sgain})
=\{cH_{ij}^{a}\mid a\in S_{ij},\ i<j\}\cup\{H_\infty\}.
\]
This is a central arrangement with hyperplane at infinity \(H_\infty\).
On the chart \(y=1\), its other hyperplanes recover the affine
arrangement.

For these deformations, the elimination ordering must respect the gain
sets. For \(A,B\subseteq\K\), write
\(A-B=\{a-b\mid a\in A,\ b\in B\}\). An ordering
\(v_1,\ldots,v_n\) is \emph{\(\Sgain\)-admissible} if, for every \(k\)
and all distinct \(i,j>k\),
\[
S_{v_kv_j}-S_{v_kv_i}\subseteq S_{v_iv_j}.
\]
If \(v_kv_i\) and \(v_kv_j\) are edges, the left-hand side is nonempty,
so \(v_iv_j\) must also be an edge. Thus admissibility includes the
perfect-elimination condition; when every edge has gain set \(\{0\}\),
the two conditions coincide.

Our first result characterizes nice partitions of the affine arrangement
by admissibility and supersolvability of the cone. We state it first for
blocks. A \emph{block} is a maximal connected subgraph without a cut
vertex; isolated vertices and bridges are included as blocks of orders
one and two, respectively.

\begin{theorem}\label{thm:main}
	Let \(\A(G_{\Sgain})\) be defined by the gain sets \(\Sgain\) as above.
	Assume that \(G\) is a block. Then the
	following statements are equivalent.
	\begin{enumerate}[label=(\arabic*)]
		\item\label{main:ss}
		The cone \(c\A(G_{\Sgain})\) is supersolvable.

		\item\label{main:ss-infty}
		The lattice \(L(c\A(G_{\Sgain}))\) has a maximal modular chain
		passing through \(H_\infty\).

		\item\label{main:nice}
		The affine arrangement \(\A(G_{\Sgain})\) admits a nice partition.

		\item\label{main:order}
		\(G_{\Sgain}\) admits an \(\Sgain\)-admissible ordering.
	\end{enumerate}
\end{theorem}

The equivalence of \ref{main:ss} and \ref{main:order} is already
contained in Zaslavsky's characterization of supersolvable graphic-lift
lattices \cite[Theorem~3.2]{Zaslavsky2001}. Indeed, let
\(\Phi_{\Sgain}\) be the gain multigraph with an edge of gain \(a\)
between \(i\) and \(j\) for every \(a\in S_{ij}\). The cone represents
the complete lift matroid \(L_0(\Phi_{\Sgain})\), and admissibility is
the set-valued form of link-simplicial elimination. See
\cite{Koban2004} for corrections to the description of modular flats
and an improved proof, and \cite{SuyamaTorielliTsujie2024} for the
application of Zaslavsky's criteria to affinographic and bias arrangements.

The affine nice-partition characterization is the new equivalence in
Theorem~\ref{thm:main}. The proof also describes individual nice
partitions: on a block, every nice partition arises from an admissible
ordering by grouping hyperplanes according to the earlier endpoint of
each defining edge. To recover this ordering, we show that the partition
induces a decomposition of the edges into stars and that independence
makes the resulting orientation acyclic. We also prove the
supersolvability assertions directly in the intersection lattice of the
cone. In particular, a supersolvable cone has a maximal modular chain
through \(H_\infty\). The partition induced by such a chain has
\(\{H_\infty\}\) as a singleton part; removing this part and passing to
\(y=1\) gives a nice partition of the affine arrangement.

For arbitrary graphs, the conditions are imposed separately on each
block. We call \(G_{\Sgain}\) \emph{blockwise admissible} if the
restricted gain graph on every block has an admissible ordering.
Up to empty factors, its intersection poset is the product of the
intersection posets of the block arrangements. We prove that modular chains
through the hyperplanes at infinity of the block cones combine to give such
a chain for the full cone. Consequently,
Corollary~\ref{cor:blockwise-general} gives
\[
\begin{aligned}
G_{\Sgain}\text{ is blockwise admissible}
&\Longleftrightarrow c\A(G_{\Sgain})\text{ is supersolvable}\\
&\Longleftrightarrow \A(G_{\Sgain})\text{ admits a nice partition}.
\end{aligned}
\]
A single admissible ordering of the whole graph is more restrictive,
because it also imposes compatibility between blocks at their cut
vertices.

A central arrangement is \emph{free} if its module of logarithmic
derivations is free over the coordinate ring. Every supersolvable
arrangement is free \cite{JambuTerao1984}; hence the equivalent conditions
above imply that \(c\A(G_{\Sgain})\) is free. For ordinary graphic
arrangements, freeness
is itself equivalent to chordality \cite{EdelmanReiner1994}. For
affine deformations, freeness need not imply supersolvability. Over a
field of characteristic zero, the Shi and Catalan cones are free
\cite{AbeTerao2011}, but for \(n\ge3\) their affine arrangements admit
no nice partition, by Theorem~\ref{thm:main} and
Example~\ref{ex:shi-catalan-linial}. Our second result proves that, for
arbitrary finite gain sets over any field, freeness of
\(c\A(G_{\Sgain})\) forces \(G\) to be chordal.

\begin{theorem}\label{thm:free-implies-chordal}
	Let \(\A(G_{\Sgain})\) be defined by the gain sets \(\Sgain\) as above.
	If the cone
	\(c\A(G_{\Sgain})\) is free, then \(G\) is chordal.
\end{theorem}

The proof uses localizations at flats contained in \(H_\infty\).
An induced cycle of length at least four gives, up to empty product
factors, the cone over the deformation supported on that cycle.
We prove that every such cone is non-free by comparing degree bounds
for logarithmic derivations with the coefficients of the characteristic
polynomial. The argument works for arbitrary finite
nonempty gain sets over every field.
Since localization preserves freeness, this proves
Theorem~\ref{thm:free-implies-chordal}.

Section~2 gives the preliminaries and the block reduction.
Sections~3 and~4 prove Theorem~\ref{thm:main}. Section~5 proves
Theorem~\ref{thm:free-implies-chordal} and discusses the Ziegler
restriction at \(H_\infty\) in connection with Terao's conjecture.
Section~6 derives the characteristic polynomial from admissible
orderings and gives examples.

\section{Preliminaries}

\subsection{Arrangements and gain graphs}

Let \(\A\) be a finite arrangement of affine hyperplanes in
\(V = \K^n\). The intersection poset \(L(\A)\) consists of \(V\)
and all nonempty intersections of hyperplanes in \(\A\), ordered by
reverse inclusion. Its minimal element is \(\hat{0}=V\). If
\(U\le W\) in \(L(\A)\), then
\([U,W]=\{Z\in L(\A)\mid U\le Z\le W\}\) denotes the interval from \(U\)
to \(W\). For \(X\in L(\A)\), write \(r(X)=\codim_V X\), and set
\(r(\A)=\max\{r(X)\mid X\in L(\A)\}\).
For \(X\in L(\A)\), the localization of \(\A\) at \(X\) is
\(\A_X=\{H\in\A\mid X\subseteq H\}\).
We call \(\A\) central if \(\bigcap_{H \in \A}H \neq \emptyset\). In
this case \(L(\A)\) is a ranked lattice. We denote its maximal element by
\(\hat{1}=\bigcap_{H\in\A}H\),
so that \(r(\A)=r(\hat{1})\).
An element of rank \(r-1\) in a ranked lattice of rank \(r\) is called a
coatom.

For a finite affine arrangement \(\A\) in \(V\), let \(\mu\) be the
M\"obius function of \(L(\A)\). Its characteristic polynomial is
\[
\chi(\A,t)=\sum_{X\in L(\A)}\mu(V,X)t^{\dim X}.
\]

When working with a central arrangement, we translate a point of
its common intersection to the origin.
Let now \(\A\) be central. Let \(\wedge\) and \(\vee\) denote the meet and
join in \(L(\A)\). An element \(X\in L(\A)\) is called modular if, for
every \(Y\in L(\A)\),
\[
r(X)+r(Y)=r(X\wedge Y)+r(X\vee Y).
\]
Equivalently, \(X\) is modular if \(X+Y\in L(\A)\) for every
\(Y\in L(\A)\); see
\cite[Corollary~2.26]{OrlikTerao1992}. A central arrangement \(\A\) of
rank \(r\) is supersolvable if \(L(\A)\) contains a maximal chain
\[
V=X_0<X_1<\cdots<X_r=\hat{1}
\]
such that every \(X_i\) is modular.

If \(\A\) is central with defining polynomial \(Q\) in a coordinate ring
\(S\), its module of logarithmic derivations is
\[
D(\A)=\{\theta\in\operatorname{Der}_{\K}(S)\mid \theta(Q)\in QS\}.
\]
Here \(\operatorname{Der}_{\K}(S)\) denotes the \(S\)-module of
\(\K\)-derivations of \(S\).
The arrangement \(\A\) is called free if \(D(\A)\) is a free \(S\)-module;
see \cite[Chapter~4]{OrlikTerao1992}.

Throughout \(G=([n],E(G))\), with \([n]=\{1,\ldots,n\}\), is a simple
graph. For a subset \(T\subseteq\K\), write \(-T=\{-a\mid a\in T\}\). Let
\(\Sgain=(S_{ij})\) be a family of finite subsets of \(\K\), indexed by
ordered pairs of distinct vertices, such that \(S_{ji}=-S_{ij}\) and
\(S_{ij}\ne\emptyset\) if and only if \(ij\in E(G)\). We write
\(G_{\Sgain}\) for the graph \(G\) together with this family of sets.
For an edge \(ij\), we write
\(\A_{ij}=\{H_{ij}^{a}\mid a\in S_{ij}\}\), the parallel class associated
with \(ij\). Thus
\(H_{ji}^{-a}=H_{ij}^{a}\), and \(\A_{ij}=\A_{ji}\).
The cone \(c\A(G_{\Sgain})\) is the arrangement in
\(\widetilde V = \K^{n+1}\), with coordinates
\((x,y)=(x_1,\ldots,x_n,y)\), defined by
\[
c\A(G_{\Sgain})
=
\{cH_{ij}^{a}\mid a\in S_{ij},\ i<j\}\cup\{H_\infty\},
\]
where
\[
cH_{ij}^{a}
=
\{(x,y)\in\K^{n+1}\mid x_i-x_j=ay\},
\]
and \(H_\infty=\{(x,y)\in\K^{n+1}\mid y=0\}\).
We call \(H_\infty\) the hyperplane at infinity.
On the affine chart \(y=1\), the hyperplanes of
\(c\A(G_{\Sgain})\setminus\{H_\infty\}\) are naturally identified with the
hyperplanes of \(\A(G_{\Sgain})\).

For subsets \(A,B\subseteq\K\), write
\(A-B=\{a-b\mid a\in A,\ b\in B\}\). Recall from the Introduction that an
ordering is \(\Sgain\)-admissible if, for every \(k\) and all distinct
\(i,j>k\),
\[
S_{v_kv_j}-S_{v_kv_i}\subseteq S_{v_iv_j}.
\]
In particular, if
both \(v_kv_i\) and \(v_kv_j\) are edges, then the left-hand side in the
defining inclusion is nonempty, so the inclusion forces
\(S_{v_iv_j}\ne\emptyset\), hence \(v_iv_j\in E(G)\). Thus the defining
condition already makes the ordering a perfect elimination ordering; in the
graphic case \(S_{ij}=\{0\}\), it is exactly the usual perfect-elimination
condition.

\begin{example}\label{ex:triangle-gain-check}
	Let \(G\) be the triangle on vertices \(1,2,3\), and test the ordering
	\(1,2,3\). The only nontrivial admissibility condition is
	\[
	S_{13}-S_{12}\subseteq S_{23}.
	\]
	Thus the gains
	\[
	S_{12}=\{0\},\qquad S_{13}=\{0,1\},\qquad S_{23}=\{0,1\}
	\]
	make the ordering admissible, while replacing \(S_{23}\) by \(\{0\}\)
	makes it fail. Geometrically, the intersection
	\(H_{12}^{0}\cap H_{13}^{1}\) is contained in \(H_{23}^{1}\); the
	admissibility condition says that this forced third hyperplane must
	already be present in the arrangement.
\end{example}

\subsection{Nice partitions and block reduction}

Let \(\A\) be an affine arrangement and let
\(\pi=(\pi_1,\ldots,\pi_s)\) be a partition of \(\A\). A
\emph{\(p\)-section} of \(\pi\) is a set
\(\{H_1,\ldots,H_p\}\) of hyperplanes of \(\A\) such that no two of
them lie in the same part of \(\pi\). Such a section is called
\emph{independent} if \(H_1\cap\cdots\cap H_p\) is nonempty of
codimension \(p\). The
partition \(\pi\) is \emph{independent} if every section of \(\pi\) is
independent.
The partition \(\pi\) is \emph{nice} if it is independent and satisfies
the following singleton condition: for
every \(X\in L(\A)\setminus\{V\}\), the partition of \(\A_X\) obtained
by intersecting the parts of \(\pi\) with \(\A_X\) and deleting the
empty intersections has a singleton part. Equivalently,
\(|\pi_i\cap \A_X|=1\) for some \(i\).
For the empty arrangement, the empty partition is regarded as nice.
A maximal modular chain in a central arrangement induces a nice partition: if
\(\A\) is central and
\[
V=X_0<X_1<\cdots<X_r=\hat{1}
\]
is a maximal chain of modular elements in \(L(\A)\), then the partition
defined, for \(i=1,\ldots,r\), by
\(\pi_i=\A_{X_i}\setminus \A_{X_{i-1}}\)
is nice \cite[Example~2.4]{Terao1992}; see also
\cite[Proposition~2.67]{OrlikTerao1992}.

We use the usual product notation for affine arrangements: if \(\A_i\) is
an affine arrangement in \(V_i\), \(i=1,2\), then
\(\A_1\times\A_2\) is the arrangement in \(V_1\oplus V_2\) consisting of
the hyperplanes \(H_1\oplus V_2\), with \(H_1\in\A_1\), and
\(V_1\oplus H_2\), with \(H_2\in\A_2\).
For block reduction, let a \emph{block} mean a maximal connected subgraph
with no cut vertex; this convention includes isolated vertices and bridges.
Let
\(G_1,\ldots,G_s\) be the blocks of \(G\),
and let \(\Sgain_\nu\) be the restriction of \(\Sgain\) to \(G_\nu\).
We say that \(G_{\Sgain}\) is \emph{blockwise admissible} if every
\((G_\nu)_{\Sgain_\nu}\) admits an \(\Sgain_\nu\)-admissible ordering.
Put \(\A_\nu=\A((G_\nu)_{\Sgain_\nu})\). Compatibility of the equations can
be checked block by block: when two blocks meet at a cut vertex, translate a
solution on one block so that the two values at that vertex agree.
Conversely, the only compatibility conditions among these equations
occur around cycles, and every cycle lies in one block. Hence
\[
L(\A(G_{\Sgain}))
\cong
L(\A_1\times\cdots\times \A_s)
\cong
L(\A_1)\times\cdots\times L(\A_s).
\]
This is the analogue, for these deformations of graphic arrangements, of
the standard product decomposition of graphic lattices;
see \cite[Lecture~4, Exercise~16]{Stanley2007}.
Products of
arrangements preserve and reflect nice partitions
\cite[Proposition~3.29]{HogeRoehrle2016}. The proof of this product result is
combinatorial and applies unchanged to the affine definition above. The
product statement reduces the affine nice-partition question to blocks. The
next lemma gives the
corresponding reduction for cones, where the shared hyperplane at infinity
replaces an ordinary product decomposition.

\begin{lemma}\label{lem:cone-product}
	Let \(\A_i\) be an affine arrangement in \(V_i\), and let
	\(H_{\infty,i}\) be the hyperplane at infinity of \(c\A_i\), for
	\(i=1,2\). Put \(\B=c(\A_1\times\A_2)\), and let
	\(H_\infty\) be the hyperplane at infinity of \(\B\).
	\begin{enumerate}[label=(\roman*)]
		\item\label{cone-product:reflection}
		If \(\B\) is supersolvable, then \(c\A_1\) and \(c\A_2\) are
		supersolvable.

		\item\label{cone-product:gluing}
		If \(L(c\A_i)\) has a maximal modular chain passing through
		\(H_{\infty,i}\) for \(i=1,2\), then \(L(\B)\) has a maximal modular
		chain passing through \(H_\infty\).
	\end{enumerate}
\end{lemma}

\begin{proof}
	Write \(\widetilde V_i=V_i\oplus\K y_i\) and
	\(\widetilde V=V_1\oplus V_2\oplus\K y\). Consider the projections
	\[
	\rho_i:\widetilde V\longrightarrow\widetilde V_i,
	\qquad
	\rho_i(x_1,x_2,y)=(x_i,y).
	\]
	The hyperplanes of \(\B\) are the pullbacks of the hyperplanes of
	\(c\A_1\) and \(c\A_2\), with
	\[
	\rho_1^{-1}(H_{\infty,1})
	=
	\rho_2^{-1}(H_{\infty,2})
	=
	H_\infty.
	\]
	Consequently, every \(Y\in L(\B)\) can be written, not necessarily uniquely,
as
	\[
	Y=\rho_1^{-1}(Y_1)\cap\rho_2^{-1}(Y_2)
	\qquad
	(Y_i\in L(c\A_i)).
	\]
	Put \(T_i=\bigcap_{H\in c\A_i}H\) and
	\(F_1=\rho_1^{-1}(T_1)\). Since \(T_1\subseteq H_{\infty,1}\), the
	flat \(F_1\) lies in \(H_\infty\) and is contained in every hyperplane
	pulled back from \(c\A_1\). If
	\(K\in c\A_2\setminus\{H_{\infty,2}\}\), then its equation has the form
	\(\alpha(x_2)-a y_2=0\) with \(\alpha\ne0\). The \(x_2\)-coordinate is
	free on \(F_1\), so \(\rho_2^{-1}(K)\) does not contain \(F_1\). Hence
	\[
	\B_{F_1}
	=
	\{\rho_1^{-1}(H)\mid H\in c\A_1\},
	\qquad
	L(\B_{F_1})\cong L(c\A_1).
	\]
	If \(\B\) is supersolvable, then so is the interval
	\(L(\B_{F_1})\) \cite[Lecture~4, Exercise~19]{Stanley2007}; thus
	\(c\A_1\) is supersolvable. Interchanging the
	factors proves \ref{cone-product:reflection}.

	For \ref{cone-product:gluing}, first suppose that
	\(M_i\in L(c\A_i)\) is modular and \(M_i\subseteq H_{\infty,i}\) for
	\(i=1,2\). Set
	\[
	M=\rho_1^{-1}(M_1)\cap\rho_2^{-1}(M_2).
	\]
	For \(Y=\rho_1^{-1}(Y_1)\cap\rho_2^{-1}(Y_2)\in L(\B)\), the fact that
	all vectors in each \(M_i\) have last coordinate zero gives
	\[
	M+Y
	=
	\rho_1^{-1}(M_1+Y_1)
	\cap
	\rho_2^{-1}(M_2+Y_2).
	\]
	Since the \(M_i\) are modular, both sums on the right belong to the
	corresponding intersection lattices. Therefore \(M+Y\in L(\B)\), and
	\(M\) is modular.
	Choose maximal modular chains
	\[
	\widetilde V_i=X_0^i<X_1^i=H_{\infty,i}
	<X_2^i<\cdots<X_{r_i}^i=T_i
	\]
in \(L(c\A_i)\). Since \(X_j^1\subseteq H_{\infty,1}\) for
\(j\ge1\), the first-line terms below have the form
\(\rho_1^{-1}(X_j^1)\cap\rho_2^{-1}(H_{\infty,2})\), so the preceding
observation applies to them as well. The chain
	\[
	\begin{aligned}
	\widetilde V
	&<\rho_1^{-1}(X_1^1)<\rho_1^{-1}(X_2^1)<\cdots
	 <\rho_1^{-1}(X_{r_1}^1)\\
	&<\rho_1^{-1}(X_{r_1}^1)\cap\rho_2^{-1}(X_2^2)<\cdots
	 <\rho_1^{-1}(X_{r_1}^1)\cap\rho_2^{-1}(X_{r_2}^2)
	\end{aligned}
	\]
	is modular by the preceding observation. Its first nontrivial element is
	\(H_\infty\), and its length is
	\(r_1+r_2-1=r(\B)\). It is therefore a maximal modular chain through
	\(H_\infty\).
\end{proof}

\section{Nice Partitions and Admissible Orderings}

This section proves the equivalence between nice partitions of
\(\A(G_{\Sgain})\) and admissible orderings of \(G_{\Sgain}\) for blocks.

\subsection{Edge partitions induced by nice partitions}

\begin{lemma}\label{lem:parallel-class-one-part}
	Let \(\pi\) be a nice partition of \(\A(G_{\Sgain})\). For every edge
	\(ij\in E(G)\), all hyperplanes in \(\A_{ij}\) belong to the same part
	of \(\pi\).
\end{lemma}

\begin{proof}
	If \(H_{ij}^{a}\) and \(H_{ij}^{b}\), with \(a\ne b\), belong to
	different parts of \(\pi\), then they would form a 2-section which is
	not independent, since
	\(H_{ij}^{a}\cap H_{ij}^{b}=\emptyset\).
\end{proof}

Accordingly, let \(\Pi:E(G)\to \pi\) send \(ij\) to the part containing
\(\A_{ij}\).

\begin{lemma}\label{lem:induced-edge-map}
	Let \(\pi\) be a nice partition of \(\A(G_{\Sgain})\), and let
	\(\Pi:E(G)\to\pi\) be the induced map on edges. Then \(\Pi\) has the
	following properties.
	\begin{enumerate}[label=(\roman*)]
		\item\label{edge-map:disjoint}
		If \(ij\) and \(kl\) are disjoint edges, then
		\(\Pi(ij)\ne \Pi(kl)\).

		\item\label{edge-map:triangle}
		If \(ik\) and \(jk\) are edges with \(\Pi(ik)=\Pi(jk)\), then
		\(ij\in E(G)\) and \(\Pi(ij)\ne \Pi(ik)\).
	\end{enumerate}
\end{lemma}

\begin{proof}
	For \ref{edge-map:disjoint}, suppose that \(ij\) and \(kl\) are
	disjoint and lie in the same part. Choose \(a\in S_{ij}\) and
	\(b\in S_{kl}\), and set \(X=H_{ij}^{a}\cap H_{kl}^{b}\). The two
	equations fix only the differences attached to the two chosen
	edges. Since the two pairs of variables are disjoint, no difference
	attached to any other edge is constant on \(X\). Hence no other
	hyperplane of \(\A(G_{\Sgain})\) contains \(X\), and
	\(\A_X=\{H_{ij}^{a},H_{kl}^{b}\}\). Since these two hyperplanes lie in
	the same part of \(\pi\), the localization \(\A_X\) has no singleton
	part, a contradiction.

	For \ref{edge-map:triangle}, choose \(a\in S_{ik}\) and
	\(b\in S_{jk}\), and put \(X=H_{ik}^{a}\cap H_{jk}^{b}\). On \(X\),
	the only additional forced difference is
	\(x_i-x_j=a-b\). Thus the only possible third hyperplane of
	\(\A(G_{\Sgain})\) containing \(X\) is \(H_{ij}^{a-b}\).
	Since \(H_{ik}^{a}\) and \(H_{jk}^{b}\) lie in the same part, the
	singleton condition forces this third hyperplane to exist and to lie in a
	different part. Thus \(ij\in E(G)\) and
	\(\Pi(ij)\ne \Pi(ik)\).
\end{proof}

\begin{lemma}\label{lem:triangle-two-parts}
	Let \(\pi\) be a nice partition of \(\A(G_{\Sgain})\), and let
	\(\Pi:E(G)\to\pi\) be the induced map on edges. For every triangle
	\(i,j,k\) in \(G\), the three edges are assigned exactly two parts:
	\[
	|\{\Pi(ij),\Pi(jk),\Pi(ik)\}|=2.
	\]
\end{lemma}

\begin{proof}
	By Lemma~\ref{lem:induced-edge-map}\ref{edge-map:triangle}, the three
	edges cannot all be assigned to the same part. If they had three
	distinct images, one hyperplane from each parallel class would give a
	3-section. But the corresponding linear forms satisfy
	\((x_i-x_j)+(x_j-x_k)=x_i-x_k\), so the three hyperplanes either have
	empty intersection or have
	intersection of codimension at most two. Thus this 3-section is not
	independent, a contradiction.
\end{proof}

\begin{lemma}\label{lem:nice-implies-chordal}
	If \(\A(G_{\Sgain})\) admits a nice partition, then \(G\) is chordal.
\end{lemma}

\begin{proof}
	Let \(\pi\) be a nice partition of \(\A(G_{\Sgain})\), and let
	\(\Pi:E(G)\to\pi\) be the induced map on edges.
	Suppose that \(G\) has a chordless cycle
	\(C=(v_1,v_2,\ldots,v_m,v_1)\)
	of length \(m\ge 4\). No two edges of \(C\) can be assigned to the same
	part. For nonadjacent edges this follows from Lemma~\ref{lem:induced-edge-map}\ref{edge-map:disjoint}; for adjacent edges it
	follows from Lemma~\ref{lem:induced-edge-map}\ref{edge-map:triangle},
	since equality of the assigned parts would force a chord.
	Choose one hyperplane from each parallel class on \(C\). These
	hyperplanes belong to distinct parts of \(\pi\), hence form an
	\(m\)-section.
	Since \(\pi\) is independent, their intersection should be nonempty of
	codimension \(m\). But the defining linear forms around a cycle are
	linearly dependent; hence the chosen hyperplanes either have empty
	intersection or have intersection of codimension at most \(m-1\). This
	contradicts independence.
\end{proof}

A star is a nonempty set of edges incident with one common vertex, called
its center.

\begin{proposition}\label{prop:star-preimages}
	Let \(\pi\) be a nice partition of \(\A(G_{\Sgain})\), and let
	\(\Pi:E(G)\to\pi\) be the induced map on edges. Assume that \(G\) is a
	block. For every part \(\pi_r\), if
	\(\Pi^{-1}(\pi_r)\) is nonempty, then it is a star. Moreover, the
	nonempty preimages have distinct centers.
\end{proposition}

\begin{proof}
	We first show that each nonempty preimage is a star. Fix a part
	\(\pi_r\), and let
	\[
	E_r=\{ij\in E(G)\mid \Pi(ij)=\pi_r\}.
	\]
	By Lemma~\ref{lem:induced-edge-map}\ref{edge-map:disjoint}, any two
	edges in \(E_r\) meet. If the edges of \(E_r\) had no common vertex, then
	three of them would form the edge set of a triangle. This is impossible
	by Lemma~\ref{lem:triangle-two-parts}, because all three edges lie in
	the same part. Thus all edges in
	\(E_r\) are incident with a common vertex; choose one such vertex and
	denote it by \(c_r\).

	It remains to show that distinct nonempty preimages have distinct
	centers. Suppose, to the contrary, that two preimages \(E_1\) and \(E_2\)
	have the same center \(v\). Let \(A=\{u\mid vu\in E_1\}\) and
	\(B=\{u\mid vu\in E_2\}\).
	By Lemma~\ref{lem:nice-implies-chordal}, \(G\) is chordal.
	Since \(G\) is a block, \(G-v\) is connected. Choose a shortest path
	\(P=(u_0,u_1,\ldots,u_k)\) in \(G-v\) from \(A\) to \(B\),
	where \(u_0\in A\) and \(u_k\in B\). By the minimality of \(P\), the
	internal vertices \(u_1,\ldots,u_{k-1}\) lie in neither \(A\) nor \(B\).
	First, \(vu_i\in E(G)\) for \(i=1,\ldots,k-1\). Indeed,
	the cycle
	\((v,u_0,u_1,\ldots,u_k,v)\)
	has length at least four unless \(k=1\). Since \(G\) is chordal and
	\(P\) is shortest in \(G-v\), any chord of this cycle must be of the
	form \(vu_i\). If the first such chord were \(vu_i\) with \(i>1\), then,
	because \(P\) is shortest in \(G-v\), the shorter cycle
	\((v,u_0,u_1,\ldots,u_i,v)\) would have no chord, contradicting
	chordality. Hence \(vu_1\in E(G)\). Repeating the same
	argument for the cycle \((v,u_1,u_2,\ldots,u_k,v)\), and then
	inducting along the path, gives \(vu_i\in E(G)\) for every internal
	vertex of \(P\).

	If \(k=1\), then \(v,u_0,u_1\) form a triangle with
	\(\Pi(vu_0)=\pi_1\) and \(\Pi(vu_1)=\pi_2\). By Lemma~\ref{lem:triangle-two-parts}, the third edge \(u_0u_1\) has image
	either \(\pi_1\) or \(\pi_2\). This is impossible, because every edge in
	either \(E_1\) or \(E_2\) is incident with \(v\), whereas \(u_0u_1\) is
	not.
	Now assume \(k\ge2\). In the triangle \(v,u_0,u_1\), the edge \(vu_0\)
	has image \(\pi_1\). The other two edges do not have image \(\pi_1\):
	\(vu_1\notin E_1\) because \(u_1\notin A\), and
	\(u_0u_1\notin E_1\) because it is not incident with \(v\). Lemma~\ref{lem:triangle-two-parts} therefore gives
	\(\Pi(u_0u_1)=\Pi(vu_1)\).
	Propagate this equality along the path. Suppose that, for some
	\(1\le t\le k-1\), \(\Pi(u_{t-1}u_t)=\Pi(vu_t)\). Consider the
	triangle \(v,u_t,u_{t+1}\). If
	\(\Pi(u_tu_{t+1})=\Pi(u_{t-1}u_t)\), then Lemma~\ref{lem:induced-edge-map}\ref{edge-map:triangle} would force the chord
	\(u_{t-1}u_{t+1}\), contradicting the minimality of \(P\). Also
	\(\Pi(vu_{t+1})\ne\Pi(u_{t-1}u_t)\), because the two edges
	\(vu_{t+1}\) and \(u_{t-1}u_t\) are disjoint, so Lemma~\ref{lem:induced-edge-map}\ref{edge-map:disjoint} applies. Hence, by
	Lemma~\ref{lem:triangle-two-parts},
	\(\Pi(u_tu_{t+1})=\Pi(vu_{t+1})\). By induction,
	\(\Pi(u_{k-1}u_k)=\Pi(vu_k)=\pi_2\).
	This is again impossible, because every edge in \(E_2\) is incident
	with \(v\), whereas \(u_{k-1}u_k\) is not.
\end{proof}

\subsection{From induced edge partitions to admissible orderings}

\begin{proposition}\label{prop:nice-star-order}
	Let \(\pi\) be a nice partition of \(\A(G_{\Sgain})\), and let
	\(\Pi:E(G)\to\pi\) be the induced map on edges. Assume that \(G\) is a
	block. After relabeling the parts, there is an ordering
	\(v_1,\ldots,v_n\)
	of the vertices such that
	\[
	\pi_i=\bigcup_{j>i}\A_{v_iv_j}
	\]
	for \(i=1,\ldots,n-1\), with empty unions omitted.
\end{proposition}

\begin{proof}
	By Proposition~\ref{prop:star-preimages}, choose pairwise distinct
	centers for the nonempty preimages.
	Orient every edge away from the chosen vertex of the preimage containing
	it. If there were
	a directed cycle
	\[
	u_1\to u_2\to\cdots\to u_m\to u_1,
	\]
	then the images under \(\Pi\) of its edges would be pairwise distinct,
	because the centers of the corresponding preimages are the distinct
	vertices \(u_1,\ldots,u_m\). Choosing one hyperplane from each
	corresponding parallel class gives an \(m\)-section of \(\pi\). The
	defining linear forms around this cycle are linearly dependent, so this
	section is not independent, a contradiction. Hence the orientation is
	acyclic.
	Take a linear extension \(v_1,\ldots,v_n\) of this acyclic orientation.
	If the chosen vertex of a preimage is \(v_i\), all its edges point from
	\(v_i\) to later vertices. Therefore the corresponding part is precisely
	\(\bigcup_{j>i}\A_{v_iv_j}\).
\end{proof}

\begin{theorem}\label{thm:nice-order}
	Assume that \(G\) is a block. The affine arrangement
	\(\A(G_{\Sgain})\) admits a nice partition if and only if
	\(G_{\Sgain}\) admits an admissible ordering.
\end{theorem}

\begin{proof}
	Suppose first that \(\A(G_{\Sgain})\) has a nice partition \(\pi\). By
	Proposition~\ref{prop:nice-star-order}, after relabeling the vertices
	we may write
	\[
	\pi_k=\bigcup_{i>k}\A_{v_kv_i}.
	\]
	Let \(i,j>k\) be distinct. If either gain set on the left-hand side of
	the admissibility inclusion is empty, there is nothing to prove.
	Otherwise, choose \(a\in S_{v_kv_i}\) and \(b\in S_{v_kv_j}\).
	The two hyperplanes \(H_{v_kv_i}^{a}\) and \(H_{v_kv_j}^{b}\) lie in
	the same part \(\pi_k\). Their intersection is contained in
	the affine hyperplane \(x_{v_i}-x_{v_j}=b-a\). If
	\(H_{v_iv_j}^{b-a}\) did not belong to the arrangement, the localization
	of the intersection would contain exactly the two chosen hyperplanes,
	both in the same part, and would have no singleton part. Hence
	\(b-a\in S_{v_iv_j}\).
	Thus the ordering is admissible.

	Conversely, assume \(v_1,\ldots,v_n\) is admissible. Define
	\(\pi_k\), for \(k=1,\ldots,n-1\), by
	\[
	\pi_k=\bigcup_{i>k}\A_{v_kv_i}
	\]
	and let \(\pi\) be the family of nonempty \(\pi_k\)'s.
	To prove independence, take a section of \(\pi\). The selected edges form
	a forest: every selected edge is oriented from the vertex indexing its
	part to a later vertex, and an undirected cycle would have a smallest
	vertex incident with two selected outgoing edges from the same part. Hence
	the defining linear forms are independent. Moreover, since the selected
	graph is a forest, the selected affine hyperplanes have nonempty
	intersection: on each tree component, choose one root coordinate freely and determine the
	remaining coordinates successively along the edges. Thus every section
	is independent.
	It remains to verify the singleton condition. Let
	\(X\in L(\A(G_{\Sgain}))\), \(X\ne V\), and choose the largest \(k\)
	such that \(\pi_k\cap \A_X\ne\emptyset\). We show that
	\[
	|\pi_k\cap \A_X|=1.
	\]
	If not, then \(X\) is contained in two distinct hyperplanes
	\(H_{v_kv_i}^{a}\) and \(H_{v_kv_j}^{b}\), with \(i,j>k\).
	Here \(i\ne j\), since two distinct parallel hyperplanes have empty
	intersection whereas \(X\ne\emptyset\).
	By admissibility, \(b-a\in S_{v_iv_j}\), so
	\(H_{v_iv_j}^{b-a}\in \A_X\).
	This hyperplane belongs to \(\pi_{\min\{i,j\}}\), and
	\(\min\{i,j\}>k\), contradicting the maximality of \(k\). Therefore the
	singleton condition holds.
\end{proof}

\section{Supersolvability of the Cone}

This section proves the implications in Theorem~\ref{thm:main} that involve
\(c\A(G_{\Sgain})\) and chains through \(H_\infty\). For a block, an
admissible ordering gives a maximal modular chain through \(H_\infty\), and
such a chain induces a nice partition of the affine arrangement. We also
prove that if \(c\A(G_{\Sgain})\) is
supersolvable, then some maximal modular chain can be chosen to pass through
\(H_\infty\).

The equivalence between supersolvability of the cone and admissibility of the
gain graph also follows from Zaslavsky's characterization of supersolvable graphic-lift lattices
\cite{Zaslavsky2001}. The proof below works directly in the intersection lattice of
\(c\A(G_{\Sgain})\); in particular, it shows how the modular chain can be
chosen to pass through \(H_\infty\) and how this choice descends to a nice
partition of the affine arrangement.

\subsection{Modular coatoms and chains}

\begin{lemma}\label{lem:coatom-criterion}
	Let \(\A\) be a central arrangement of rank \(r\), and let
	\(X\in L(\A)\) have rank \(r-1\). Then \(X\) is modular if and only if
	for every pair of distinct hyperplanes \(H_1,H_2\notin \A_X\), there exists
	\(H'\in\A_X\) such that
	\[
	H_1\cap H_2\subseteq H'.
	\]
\end{lemma}

\begin{proof}
	First suppose that \(X\) is modular. Let
	\(H_1,H_2\notin \A_X\) be distinct hyperplanes, and put
	\(Y=H_1\vee H_2=H_1\cap H_2\). Since \(X\) has rank \(r-1\) and neither
	\(H_1\) nor \(H_2\) contains \(X\), we have \(X\vee Y=\hat{1}\). Also
	\(r(Y)=2\). Hence the modular rank identity gives
	\[
	r(X\wedge Y)
	=
	r(X)+r(Y)-r(X\vee Y)
	=
	(r-1)+2-r
	=
	1.
	\]
	Thus \(X\wedge Y\) is a rank-one element of \(L(\A)\), hence a
	hyperplane \(H'\) of \(\A\). Therefore \(H'\in\A_X\) and
	\(H_1\cap H_2=Y\subseteq H'\).

	Conversely, assume that the stated condition holds. By Stanley's
	characterization of modular elements in a geometric lattice, one of the
	equivalent conditions for \(X\) to be modular is that all complements of
	\(X\) are incomparable
	\cite[Proposition~4.10(a)]{Stanley2007}. Hence it is enough to show that
	all complements of \(X\) in \(L(\A)\) are incomparable. Let \(Z\) be a
	complement of \(X\), so that
	\(
	X\wedge Z=\hat{0},
	\;
	X\vee Z=\hat{1}
	\).
	We show that \(r(Z)=1\). Suppose, to the contrary, that \(r(Z)\ge 2\).
	Then \(\A_Z\) contains two distinct hyperplanes \(H_1,H_2\). Moreover,
	no hyperplane in \(\A_Z\) can belong to \(\A_X\); otherwise it would be a
	nonzero lower bound of \(X\) and \(Z\), contradicting
	\(X\wedge Z=\hat{0}\). Hence \(H_1,H_2\notin\A_X\). By the hypothesis,
	there exists \(H'\in\A_X\) such that
	\(H_1\cap H_2\subseteq H'\).
	Since \(Z\subseteq H_1\cap H_2\), this gives \(Z\subseteq H'\), and so
	\(H'\in\A_Z\cap\A_X\), again contradicting \(X\wedge Z=\hat{0}\).
	Therefore every complement \(Z\) of \(X\) has rank one. Distinct
	rank-one elements are incomparable, so all complements of \(X\) are
	incomparable.
\end{proof}

\begin{lemma}\label{lem:lifting}
	Let \(\A\) be a central arrangement, and let \(X\in L(\A)\) be a modular
	coatom. If \(L(\A_X)\) has a maximal modular chain passing through some
	\(H\in\A_X\), then \(L(\A)\) has a maximal modular chain passing through
	\(H\).
\end{lemma}

\begin{proof}
	Let
	\[
	V=X_0<X_1<\cdots<X_{r-1}=X
	\]
	be a maximal modular chain in the interval \([V,X]\cong L(\A_X)\) with
	\(X_1=H\). Since \(X\) is modular in \(L(\A)\), modularity in the
	interval \([V,X]\) lifts to modularity in \(L(\A)\) by
	\cite[Proposition~4.10(b)]{Stanley2007}. Thus
	\[
	V=X_0<X_1<\cdots<X_{r-1}=X<X_r=\hat{1}
	\]
	is a maximal modular chain of \(L(\A)\) passing through \(H\).
\end{proof}

\subsection{Admissible orderings and nice partitions}

\begin{proposition}\label{prop:order-to-ss-infty}
	Assume that \(G\) is a block. If \(G_{\Sgain}\) has an admissible
	ordering, then
	\(L(c\A(G_{\Sgain}))\) has a maximal modular chain passing through
	\(H_\infty\).
\end{proposition}

\begin{proof}
	Let \(v_1,\ldots,v_n\) be an admissible ordering. We argue by induction
	on \(n\). If \(n=1\), then \(\B=\{H_\infty\}\), whose intersection
	lattice has the required modular chain. Assume henceforth that \(n\ge2\),
	and put \(\B=c\A(G_{\Sgain})\).

	Since \(G\) is a block, \(G-v_1\) is connected. In fact, the induced
	subgraph \(G'=G-\{v_1\}\) is again a block, with the one-vertex case
	harmless. Indeed, if a vertex \(u\) were a cut vertex of \(G'\), then all
	neighbors of \(v_1\) different from \(u\) would lie in a single
	component of \(G'-u\), because these neighbors form a clique. Adding
	\(v_1\) back therefore cannot reconnect all components of \(G'-u\).
	Hence \(u\) would be a cut vertex of \(G\), a contradiction. Set
	\[
	X=
	H_\infty\cap
	\{x_{v_2}=x_{v_3}=\cdots=x_{v_n}\}.
	\]
	Choose a spanning tree of \(G-v_1\). Since on \(H_\infty\) every
	hyperplane \(cH_{ij}^{a}\) is given by \(x_i=x_j\), the flat \(X\) is
	obtained by intersecting \(H_\infty\) with hyperplanes corresponding to
	the edges of this tree, choosing one gain \(a_{ij}\in S_{ij}\) for each
	tree edge \(ij\). Hence \(X\in L(\B)\). Since \(G-v_1\) is connected,
	\(X\) has rank \(n-1\), and therefore \(X\) is a coatom of \(L(\B)\).
	The hyperplanes of \(\B\) not containing \(X\) are precisely the
	hyperplanes \(cH_{v_1v_i}^{a}\), where \(i>1\) and
	\(a\in S_{v_1v_i}\). Take two such hyperplanes. If they belong to the
	same edge and are defined by distinct elements of \(S_{v_1v_i}\), then
	their intersection is contained in \(H_\infty\), which contains \(X\).
	If they are
	\(cH_{v_1v_i}^{a}\) and \(cH_{v_1v_j}^{b}\) with \(i\ne j\), then
	admissibility gives \(b-a\in S_{v_iv_j}\), and
	\[
	cH_{v_1v_i}^{a}\cap cH_{v_1v_j}^{b}
	\subseteq
	cH_{v_iv_j}^{b-a}.
	\]
	The hyperplane \(cH_{v_iv_j}^{b-a}\) contains \(X\). By Lemma~\ref{lem:coatom-criterion}, \(X\) is modular in \(L(\B)\).

	Let \(G'\) be the subgraph induced by \(v_2,\ldots,v_n\), and let
	\(\Sgain'\) be the restriction of \(\Sgain\) to \(G'\). The localization
	\(\B_X\) has the same intersection lattice as
	\(c\A(G'_{\Sgain'})\), after forgetting the coordinate \(x_{v_1}\),
	which does not appear in the remaining defining equations.
	The ordering \(v_2,\ldots,v_n\) is admissible for \(G'_{\Sgain'}\).
	By induction, \(L(\B_X)\) has a maximal modular chain passing through
	\(H_\infty\).
	Lemma~\ref{lem:lifting} lifts the corresponding chain to \(L(\B)\), so
	\(L(\B)\) has a maximal modular chain passing through \(H_\infty\).
\end{proof}

\begin{proposition}\label{prop:ss-infty-to-nice}
	If \(L(c\A(G_{\Sgain}))\) has a maximal modular chain passing through
	\(H_\infty\), then \(\A(G_{\Sgain})\) admits a nice partition.
\end{proposition}

\begin{proof}
	Let
	\[
	\widetilde V=X_0<X_1<\cdots<X_r=\hat{1}
	\]
	be a maximal modular chain in \(L(c\A(G_{\Sgain}))\) with
	\(X_1=H_\infty\). Let \(\pi=(\pi_1,\ldots,\pi_r)\) be the partition of
	\(c\A(G_{\Sgain})\) induced by this chain. The modular-chain
	construction recalled in Section 2 makes \(\pi\) nice
	\cite[Example~2.4]{Terao1992}; see also
	\cite[Proposition~2.67]{OrlikTerao1992}. Since \(X_1=H_\infty\), we
	have \(\pi_1=\{H_\infty\}\). Remove the part \(\pi_1\) and pass the
	remaining hyperplanes to the affine chart \(y=1\). This gives a partition
	\(\bar\pi\) of
	\(\A(G_{\Sgain})\).

	For independence, choose hyperplanes \(H_1,\ldots,H_p\), one from each
	of \(p\) distinct parts of \(\bar\pi\). Their cones
	\(cH_1,\ldots,cH_p\) lie in distinct parts of \(\pi\), so
	\[
	cH_1\cap\cdots\cap cH_p
	\]
	has codimension \(p\). This intersection is not contained in
	\(H_\infty\); otherwise adding \(H_\infty\) would give a section of
	\(\pi\) of size \(p+1\) whose intersection still has codimension \(p\),
	contrary to the independence of \(\pi\). Hence the affine intersection
	\(H_1\cap\cdots\cap H_p\) is nonempty of codimension \(p\). Thus
	\(\bar\pi\) is independent.

	For the singleton condition, let
	\(X\in L(\A(G_{\Sgain}))\) with \(X\ne\K^n\), and let
	\(\widetilde X\in L(c\A(G_{\Sgain}))\) be the corresponding element not
	contained in \(H_\infty\). Since \(\pi\) is nice, there is a part
	\(\pi_i\) such that
	\[
	|\pi_i\cap (c\A(G_{\Sgain}))_{\widetilde X}|=1.
	\]
	As \(\widetilde X\nsubseteq H_\infty\), this part is not
	\(\pi_1=\{H_\infty\}\). Passing its unique hyperplane to the affine chart
	\(y=1\) gives a singleton part in the localization of
	\(\A(G_{\Sgain})\) at \(X\).
	Therefore \(\bar\pi\) is nice.
\end{proof}

\begin{corollary}\label{cor:ss-infty-nice-order}
	Assume that \(G\) is a block.
	The following conditions are equivalent:
	\begin{enumerate}[label=(\arabic*)]
		\item \(L(c\A(G_{\Sgain}))\) has a maximal modular chain passing
		through \(H_\infty\);
		\item \(\A(G_{\Sgain})\) admits a nice partition;
		\item \(G_{\Sgain}\) has an admissible ordering.
	\end{enumerate}
\end{corollary}

\begin{proof}
	This follows from Theorem~\ref{thm:nice-order}, Proposition~\ref{prop:order-to-ss-infty}, and Proposition~\ref{prop:ss-infty-to-nice}.
\end{proof}

\subsection{Forcing the chain through \texorpdfstring{\(H_\infty\)}{H-infinity}}

\begin{lemma}\label{lem:proper-clique}
	Let \(G\) be a chordal graph and let \(K\) be a proper clique of \(G\).
	Then \(G\) has a simplicial vertex outside \(K\), that is, a vertex
	whose neighbors form a clique.
\end{lemma}

\begin{proof}
	If \(G\) is complete, every vertex is simplicial, and any vertex in
	\(V(G)\setminus K\) has the required property. Otherwise, by Dirac's
	theorem, every non-complete chordal graph has two nonadjacent simplicial
	vertices \cite{Dirac1961}. Since \(K\) is a clique, two nonadjacent
	vertices cannot both lie in \(K\). Hence at least one of these
	simplicial vertices lies outside \(K\).
\end{proof}

We record a normalization used in the next proof. Let
\(\tau=(t_1,\ldots,t_n)\in\K^n\), and define a new family
\(\Sgain^\tau=(S_{ij}^\tau)\) by
\[
S_{ij}^\tau=\{a-t_i+t_j\mid a\in S_{ij}\}.
\]
Consider the linear automorphism \(\Phi_\tau\) of \(\K^{n+1}\) given by
\[
\Phi_\tau(x_1,\ldots,x_n,y)
=
(x_1-t_1y,\ldots,x_n-t_ny,y).
\]
It fixes \(H_\infty\) pointwise. Moreover, if \(a\in S_{ij}\), then
\(\Phi_\tau\) sends \(cH_{ij}^{a}\) onto \(cH_{ij}^{a-t_i+t_j}\).
Hence \(\Phi_\tau\) carries the arrangement \(c\A(G_{\Sgain})\) onto
\(c\A(G_{\Sgain^\tau})\), and therefore induces an isomorphism
\[
L(c\A(G_{\Sgain}))\cong L(c\A(G_{\Sgain^\tau}))
\]
which fixes the element \(H_\infty\). Thus supersolvability, modularity,
and the existence of a maximal modular chain passing through
\(H_\infty\) are unchanged after replacing \(\Sgain\) by
\(\Sgain^\tau\).

It remains to prove that, for a block, supersolvability of the cone forces
the existence of a maximal modular chain through \(H_\infty\).

\begin{theorem}\label{thm:ss-implies-infty}
	Let \(\A(G_{\Sgain})\) be defined as in Section 2, and assume that
	\(G\) is a block. If
	\(c\A(G_{\Sgain})\) is supersolvable, then
	\(L(c\A(G_{\Sgain}))\) has a maximal modular chain passing through
	\(H_\infty\).
\end{theorem}

\begin{proof}
	\emph{Step 1: reduce to a modular coatom contained in \(H_\infty\).}
	We prove the assertion simultaneously for all gain graphs whose underlying
	graph is a block, by strong induction on the rank of the cone. Put
	\(\B=c\A(G_{\Sgain})\). In rank at most two, every hyperplane is modular, so
	\(H_\infty\) can be placed in a maximal modular chain. Hence assume
	\(r(\B)>2\).
	It is enough to find a modular coatom
	\(X\in L(\B)\) such that \(X\subseteq H_\infty\).
	Since \(X\subseteq H_\infty\), whether a cone hyperplane contains
	\(X\) depends only on its underlying edge, not on its gain.
	Hence \(\B_X\) is the cone associated
	with the restricted gain graph \(F_{\Sgain|_F}\), where \(F\) may be
	disconnected and may have cut vertices, and it
	has rank \(r(\B)-1\). Moreover, \(L(\B_X)\) is an interval in the
	supersolvable lattice \(L(\B)\), so it is supersolvable
	\cite[Lecture~4, Exercise~19]{Stanley2007}.
	Let \(F_1,\ldots,F_t\) be the blocks of \(F\), and put
	\(\A_\mu^F=\A((F_\mu)_{\Sgain|_{F_\mu}})\). By the block-product
	decomposition in Section~2, there is an isomorphism
	\[
	L(\B_X)
	\cong
	L\bigl(c(\A_1^F\times\cdots\times\A_t^F)\bigr)
	\]
	which preserves the hyperplane at infinity. Thus the cone of the product
	on the right is supersolvable. Iterating Lemma~\ref{lem:cone-product}\ref{cone-product:reflection} shows that every
	block cone \(c\A_\mu^F\) is supersolvable. Each block cone has rank at
	most \(r(\B_X)=r(\B)-1\), so the strong induction hypothesis applies to
	each block \(F_\mu\) and gives a maximal modular chain in
	\(L(c\A_\mu^F)\) through its hyperplane at infinity.
	Iterating Lemma~\ref{lem:cone-product}\ref{cone-product:gluing} combines these block
	chains into a maximal modular chain through the common hyperplane at
	infinity in
	\(L(c(\A_1^F\times\cdots\times\A_t^F))\). Transporting it through the
	isomorphism above gives such a chain in \(L(\B_X)\). Finally, Lemma~\ref{lem:lifting} lifts that chain to \(L(\B)\).

	\emph{Step 2: normalize a modular coatom not contained in \(H_\infty\).}
	Since \(\B\) is
	supersolvable, choose a maximal modular chain
	\[
	\widetilde V=X_0<X_1<\cdots<X_{r-1}=Y<X_r=\hat{1}.
	\]
	If \(Y\subseteq H_\infty\), we take \(X=Y\). Hence assume that
	\(Y\nsubseteq H_\infty\).
	Since \(H_\infty\notin\B_Y\), we have
	\(Y\wedge H_\infty=\widetilde V\). Modularity of \(Y\) then gives
	\(r(Y\vee H_\infty)=r(\B)\), so \(Y\vee H_\infty=\hat{1}\). Thus
	\(Y\) and \(H_\infty\) are complements, and the standard modular interval
	isomorphism is
	\[
	[\widetilde V,Y]\cong [H_\infty,\hat{1}].
	\]
	In lattice notation, this map is \(Z\mapsto Z\vee H_\infty\), with
	inverse \(W\mapsto Y\wedge W\).
	The latter interval is the intersection lattice of the graphic
	arrangement of \(G\). Its atoms correspond to the edges of \(G\), while
	the atoms of \([\widetilde V,Y]\) are the hyperplanes in \(\B_Y\).
	Therefore \(\B_Y\) contains exactly one hyperplane from each parallel class
	\(\{cH_{ij}^{a}\mid a\in S_{ij}\}\).
	Write this distinguished hyperplane as \(cH_{ij}^{a_{ij}}\).
	Since \(Y\nsubseteq H_\infty\), choose a point
	\((t_1,\ldots,t_n,1)\in Y\). For every edge \(ij\in E(G)\), the unique
	hyperplane in the parallel class \(\{cH_{ij}^{a}\mid a\in S_{ij}\}\)
	which contains \(Y\) is \(cH_{ij}^{a_{ij}}\), where
	\(a_{ij}=t_i-t_j\). Replace each set \(S_{ij}\) by
	\(S_{ij}-a_{ij}\).
	This replacement is induced by the linear automorphism
	\[
	(x_1,\ldots,x_n,y)\mapsto
	(x_1-t_1y,\ldots,x_n-t_ny,y),
	\]
	which fixes \(H_\infty\) and preserves the intersection lattice of the
	cone. Hence we may assume, from now on, that
	\[
	\B_Y=\{cH_{ij}^{0}\mid ij\in E(G)\}.
	\]
	\emph{Step 3: extract the restrictions forced by modularity.}
	Let
	\[
	M=\{ij\in E(G)\mid S_{ij}\ne\{0\}\}.
	\]
	If \(ij,kl\in M\) are disjoint, choose nonzero
	\(a\in S_{ij}\) and \(b\in S_{kl}\). Then
	\(cH_{ij}^{a}\cap cH_{kl}^{b}\)
	is not contained in any zero hyperplane \(cH_{pq}^{0}\). Indeed, such a
	containment would imply the linear-form relation
	\[
	x_p-x_q\in
	\operatorname{span}_{\K}
	\{x_i-x_j-ay,\ x_k-x_l-by\}.
	\]
	Because the two edges are disjoint, comparison of the four vertex
	coefficients shows that a nonzero combination on the right either has
	support on one of the two edges or on all four vertices. In the first
	case its \(y\)-coefficient is nonzero since \(a,b\ne0\), and in the
	second case it cannot equal a two-variable form \(x_p-x_q\). This is a
	contradiction. Thus no zero hyperplane contains the intersection, contrary
	to Lemma~\ref{lem:coatom-criterion}; hence any two edges in \(M\) meet.

	Next suppose that \(ij,ik\in M\). Choose nonzero
	\(a\in S_{ij}\) and \(b\in S_{ik}\). The intersection
	\[
	cH_{ij}^{a}\cap cH_{ik}^{b}
	\]
	can be contained in a zero hyperplane from \(\B_Y\) only if
	\(jk\in E(G)\) and \(a=b\): on the affine chart \(y=1\), the two
	equations force \(x_j-x_k=b-a\), and the only possible forced equality
	among the three vertices is \(x_j=x_k\). Equivalently, for any two
	edges of \(M\) meeting at a vertex \(p\),
	\[
	\tag{*}\label{eq:local-M}
	\begin{gathered}
		pq,pr\in M,\quad
		\alpha\in S_{pq}\setminus\{0\},\quad
		\beta\in S_{pr}\setminus\{0\}  \\
		\Longrightarrow\quad
		qr\in E(G)\ \text{and}\ \alpha=\beta.
	\end{gathered}
	\]
	Together with the preceding paragraph, this shows that the set \(K\) of
	vertices incident with edges of \(M\) spans a clique. Indeed, for
	\(u,v\in K\), choose incident edges \(e,f\in M\). If either is \(uv\),
	there is nothing to prove; otherwise the pairwise-intersection property
	forces \(e=up\) and \(f=vp\) for some \(p\), and
	\eqref{eq:local-M} supplies the edge \(uv\).
	The interval \([\widetilde V,Y]\) is supersolvable, hence so is
	\([H_\infty,\hat{1}]\). Since this interval is the lattice of the graphic
	arrangement of \(G\), Stanley's theorem implies that \(G\) is chordal.

	\emph{Step 4: construct the desired modular coatom.}
	Assume first that \(K\ne V(G)\). By Lemma~\ref{lem:proper-clique},
	there is a simplicial vertex \(v\notin K\). Then \(S_{vj}=\{0\}\) for
	every edge \(vj\). Define
	\[
	X_v=
	H_\infty
	\cap
	\bigcap_{\substack{jk\in E(G)\\ j,k\ne v}} cH_{jk}^{0}.
	\]
	Because \(G\) is a block, \(G-v\) is connected; the displayed equations
	therefore have rank \(n-1\), so \(X_v\) is a coatom contained in
	\(H_\infty\). The hyperplanes not containing \(X_v\) are precisely the
	hyperplanes \(cH_{vj}^{0}\). If \(cH_{vj}^{0}\) and \(cH_{vk}^{0}\) are
	two such hyperplanes with \(j\ne k\), then \(jk\in E(G)\), and
	\[
	cH_{vj}^{0}\cap cH_{vk}^{0}\subseteq cH_{jk}^{0}\in \B_{X_v}.
	\]
	Thus \(X_v\) is modular by Lemma~\ref{lem:coatom-criterion}.

	It remains to consider \(K=V(G)\). Since any two edges in \(M\) meet,
	either \(M\) is a star or \(M\) consists of the three edges of a triangle.

	\begin{itemize}[leftmargin=2em]
		\item \emph{Star case.}
		Suppose that \(M\) is a star with center \(q\). Choose a leaf \(v\), and
		note that \(G\) is complete. If the star has at least two leaves, then
		applying \eqref{eq:local-M} to any two edges \(qj,qk\in M\) gives
		\(jk\in E(G)\); if it has only one leaf, this is the two-vertex
		case. Hence \(v\) is simplicial. Define
		\[
		X_v=
		H_\infty
		\cap
		\bigcap_{\substack{jk\in E(G)\\ j,k\ne v}} cH_{jk}^{0}.
		\]
		As above, \(X_v\) is a coatom contained in \(H_\infty\). The hyperplanes
		not containing \(X_v\) are precisely the hyperplanes \(cH_{vj}^{a}\).
		If \(j\ne q\), then \(vj\notin M\), and hence \(S_{vj}=\{0\}\).

		Apply Lemma~\ref{lem:coatom-criterion}. Pairs from
		the same class \(\A_{vj}\) have intersection contained in
		\(H_\infty\). If \(j,k\ne q\), then
		\[
		cH_{vj}^{0}\cap cH_{vk}^{0}\subseteq cH_{jk}^{0}\in\B_{X_v}.
		\]
		It remains to consider \(cH_{vq}^{a}\) and \(cH_{vj}^{0}\), where
		\(j\ne q\). If \(a=0\), their intersection is contained in
		\(cH_{qj}^{0}\). If \(a\ne0\), then \(-a\in S_{qv}\). Since
		\(qv,qj\in M\), choose \(\beta\in S_{qj}\setminus\{0\}\). Applying
		\eqref{eq:local-M} at \(q\) to \(-a\in S_{qv}\) and
		\(\beta\in S_{qj}\) gives \(\beta=-a\). Hence \(-a\in S_{qj}\), and
		\[
		cH_{vq}^{a}\cap cH_{vj}^{0}\subseteq cH_{qj}^{-a}\in\B_{X_v}.
		\]
		Therefore \(X_v\) is modular.

		\item \emph{Triangle case.}
		Suppose that \(M\) is not a star. Then \(G\) has exactly three vertices.
		Choose a vertex \(v\), and let \(u,w\) be the other two vertices. We
		first show that
		\[
		S_{vw}-S_{vu}\subseteq S_{uw}.
		\]
		Let \(a\in S_{vu}\) and \(b\in S_{vw}\). If \(a=b=0\), the claim is
		immediate. If both \(a\) and \(b\) are nonzero, then
		\eqref{eq:local-M} applied at \(v\) gives \(a=b\), so
		\(b-a=0\in S_{uw}\). If \(a=0\) and \(b\ne0\), apply
		\eqref{eq:local-M} at \(w\) to the nonzero elements
		\(-b\in S_{wv}\) and an element of \(S_{wu}\setminus\{0\}\); this gives
		\(b\in S_{uw}\). If \(a\ne0\) and \(b=0\), the same argument at \(u\)
		gives \(-a\in S_{uw}\). Thus \(b-a\in S_{uw}\) in all cases.

		Put
		\[
		X_v=H_\infty\cap cH_{uw}^{0}.
		\]
		Then \(X_v\) is a coatom contained in \(H_\infty\). The hyperplanes not
		containing \(X_v\) belong to the two classes \(\A_{vu}\) and
		\(\A_{vw}\). Pairs from the same class have intersection contained in
		\(H_\infty\). For \(cH_{vu}^{a}\) and \(cH_{vw}^{b}\), we have
		\[
		cH_{vu}^{a}\cap cH_{vw}^{b}
		\subseteq
		cH_{uw}^{b-a}.
		\]
		Since \(b-a\in S_{uw}\), the last hyperplane belongs to \(\B_{X_v}\).
		Hence \(X_v\) is modular by Lemma~\ref{lem:coatom-criterion}.
	\end{itemize}
	In all cases we have found a modular coatom contained in \(H_\infty\).
	By the induction argument at the beginning of the proof, the theorem
	follows.
\end{proof}

\begin{proof}[Proof of Theorem~\ref{thm:main}]
	The implication
	\ref{main:ss-infty}\(\Rightarrow\)\ref{main:ss} is immediate, and
	\ref{main:ss}\(\Rightarrow\)\ref{main:ss-infty} is Theorem~\ref{thm:ss-implies-infty}. The equivalence of
	\ref{main:ss-infty}, \ref{main:nice}, and \ref{main:order} is Corollary~\ref{cor:ss-infty-nice-order}.
\end{proof}

\begin{corollary}\label{cor:blockwise-general}
	Let \(G\) be an arbitrary graph with blocks \(G_1,\ldots,G_s\), and let
	\(\Sgain_\nu\) be the restriction of \(\Sgain\) to \(G_\nu\). Then the
	following conditions are equivalent:
	\begin{enumerate}[label=(\arabic*)]
		\item The cone \(c\A(G_{\Sgain})\) is supersolvable.

		\item The lattice \(L(c\A(G_{\Sgain}))\) has a maximal modular chain
		passing through \(H_\infty\).

		\item \(\A(G_{\Sgain})\) admits a nice partition.

		\item \(G_{\Sgain}\) is blockwise admissible, that is, every
		\((G_\nu)_{\Sgain_\nu}\) admits an \(\Sgain_\nu\)-admissible
		ordering.
	\end{enumerate}
\end{corollary}

\begin{proof}
	Choose coordinates along the block--cut forest. If \(G\) has \(c\)
	connected components and \(\varnothing_d\) denotes the empty arrangement
	in \(\K^d\), then
	\[
	\A(G_{\Sgain})\times\varnothing_{s-c}
	\cong
	\A_1\times\cdots\times\A_s,
	\qquad
	\A_\nu=\A((G_\nu)_{\Sgain_\nu}).
	\]
	Empty factors do not change an intersection lattice, either before or
	after coning.
	Products preserve and reflect nice partitions
	\cite[Proposition~3.29]{HogeRoehrle2016}; hence
	\(\A(G_{\Sgain})\) is nice if and only if every \(\A_\nu\) is nice.
	By Theorem~\ref{thm:main}, this is equivalent to blockwise
	admissibility and to supersolvability of every block cone.

	If the full cone is supersolvable, iterating Lemma~\ref{lem:cone-product}\ref{cone-product:reflection} shows that every
	block cone is supersolvable. Conversely, if every block cone is
	supersolvable, Theorem~\ref{thm:main} gives each one a maximal modular
	chain through its hyperplane at infinity. Iterating Lemma~\ref{lem:cone-product}\ref{cone-product:gluing} produces a maximal
	modular chain of \(L(c\A(G_{\Sgain}))\) through \(H_\infty\). This proves
	all four equivalences.
\end{proof}

\begin{remark*}
	The blockwise condition in Corollary~\ref{cor:blockwise-general} cannot
	in general be replaced by a single \(\Sgain\)-admissible ordering of
	\(G_{\Sgain}\). A global ordering imposes additional compatibility
	between the admissible orderings of blocks meeting at a cut vertex.
\end{remark*}

\section{Freeness and Chordality}

Suyama, Torielli, and Tsujie studied the correspondence between freeness
properties of the cone over an affinographic arrangement and those of the
associated bias arrangement \cite{SuyamaTorielliTsujie2024}. Here we obtain
a graph-theoretic necessary condition: freeness of \(c\A(G_{\Sgain})\) forces
chordality of \(G\), although it need not force gain-admissibility.
The proof reduces the obstruction to cycles. An induced
chordless cycle \(C\) gives a localization of \(c\A(G_{\Sgain})\) that is,
up to empty product factors, the cone over the arrangement supported on
\(C\). Thus the key point is to show that such cones over cycles are never free.

\subsection{Cones over cycles}

\begin{lemma}\label{lem:cycle-not-free}
	Let \(C_m\) be a cycle of length \(m\ge4\), and let
	\(\Sgain=(S_i)_{i=1}^m\) be a family of gain sets on its oriented
	edges, with each \(S_i\) nonempty and finite. Then the cone
	\(c\A((C_m)_{\Sgain})\) is not free.
\end{lemma}

\begin{proof}
	Label the vertices cyclically as \(1,\ldots,m\), and orient the edges as
	\(1\to2\to\cdots\to m\to1\). Put
	\(u_i=x_i-x_{i+1}\) for \(1\leq i<m\), and set \(t=x_m\). Then
	\[
	\K[x_1,\ldots,x_m,y]
	=
	\K[u_1,\ldots,u_{m-1},t,y],
	\qquad
	u_m:=x_m-x_1=-u_1-\cdots-u_{m-1}.
	\]
	Indeed, \(x_i=t+u_i+\cdots+u_{m-1}\) for \(1\leq i<m\).
	All defining equations of the cone are independent of \(t\). Thus the
	cone is the product of the arrangement described below with the empty
	arrangement in the \(t\)-direction. By the product property for freeness,
	it suffices to work over
	\(R=\K[u_1,\ldots,u_{m-1},y]\).
	For \(1\leq i\leq m\), put
	\(Q_i=\prod_{a\in S_i}(u_i-ay)\) and \(n_i=|S_i|\).
	Let \(\B\) be the central arrangement in \(\K^m\) with defining
	polynomial \(Q=yQ_1\cdots Q_m\), and set \(D=D(\B)\).
	Reindex the pairs \((u_i,S_i)\) so that
	\(n_1\leq\cdots\leq n_m\). The relation
	\(u_1+\cdots+u_m=0\) is unchanged, and any \(m-1\) of the
	\(u_i\) remain linearly independent.

	Suppose that \(\B\) is free. Its hyperplanes intersect only at the
	origin, so every nonzero homogeneous element of \(D\) has positive
	polynomial degree. Here the polynomial degree of a homogeneous
	derivation is the common degree of its nonzero coefficients.
	The Euler derivation
	\(\theta_E=y\partial_y+\sum_{i=1}^{m-1}u_i\partial_{u_i}\)
	can be included in a homogeneous basis of \(D\): its expansion in
	any such basis involves only constant coefficients on degree-one
	basis elements. Replacing every remaining basis element \(\theta\)
	by \(\theta-(\theta(y)/y)\theta_E\), we obtain a homogeneous basis
	\(\theta_E,\delta_1,\ldots,\delta_{m-1}\) with
	\(\delta_j(y)=0\). Write \(d_j=\deg\delta_j\), ordered so that
	\(d_1\leq\cdots\leq d_{m-1}\).
	Let \(M\) be the matrix whose columns are the coefficients of this
	basis relative to
	\(\partial_y,\partial_{u_1},\ldots,\partial_{u_{m-1}}\).
	Since \(Q\partial_y,Q\partial_{u_1},\ldots,Q\partial_{u_{m-1}}\)
	belong to \(D\), expressing them in this basis gives a matrix
	\(C\) over \(R\) with \(MC=QI_m\). Thus
	\(\det(M)\det(C)=Q^m\ne0\), so \(\det(M)\ne0\).
	For each \(i\) and \(a\in S_i\), the logarithmic condition gives
	\(u_i-ay\mid\delta_j(u_i-ay)=\delta_j(u_i)\).
	For fixed \(i\), these linear forms are pairwise coprime, so
	\(Q_i\mid\delta_j(u_i)\). In particular,
	\(\delta_j(u_i)=0\) whenever \(d_j<n_i\). We deduce that
	\[
	d_j\geq n_{j+1}\qquad(1\leq j\leq m-1).
	\]
	Indeed, if \(d_j<n_{j+1}\), then
	\(\delta_k(u_i)=0\) for \(k\leq j\) and \(i\geq j+1\).
	Applying \(\delta_k\) to \(u_1+\cdots+u_m=0\) gives
	\(\delta_k(u_j)=-\sum_{i=1}^{j-1}\delta_k(u_i)\).
	Since \(\delta_k(y)=0\), we obtain
	\[
	\delta_k=\sum_{i=1}^{j-1}\delta_k(u_i)
	(\partial_{u_i}-\partial_{u_j})
	\qquad(1\leq k\leq j).
	\]
	The corresponding \(j\) columns of \(M\) are therefore
	\(R\)-linear combinations of \(j-1\) fixed columns.
	Multilinearity of the determinant gives \(\det(M)=0\), a
	contradiction. For \(j=1\), the displayed sum is empty and
	\(\delta_1=0\), which gives the same contradiction.

	Write \(\chi(\B,t)=t^m-b_1t^{m-1}+b_2t^{m-2}+\cdots\).
	Since \(m\geq4\), any three distinct \(u_i\) are linearly
	independent. The rank-two flats are therefore
	\(X_i=\{y=u_i=0\}\), with \(|\B_{X_i}|=n_i+1\), and
	intersections of hyperplanes from distinct classes, each contained
	in exactly two hyperplanes. Using
	\(\mu(\K^m,X)=|\B_X|-1\), we obtain
	\[
	b_1=1+\sum_{i=1}^m n_i,
	\qquad
	b_2=\sum_{i=1}^m n_i+\sum_{1\leq i<k\leq m}n_in_k.
	\]
	Freeness gives
	\(\chi(\B,t)=(t-1)\prod_{j=1}^{m-1}(t-d_j)\)
	\cite[Chapter~4]{OrlikTerao1992}.
	Comparing the first two coefficients yields
	\[
	\sum_{j=1}^{m-1}d_j=\sum_{i=1}^m n_i,
	\qquad
	\sum_{j=1}^{m-1}d_j^2=\sum_{i=1}^m n_i^2.
	\]
	However, \(d_j\geq n_{j+1}\geq n_1>0\) and
	\(\sum_{j=1}^{m-1}(d_j-n_{j+1})=n_1\), so
	\[
	\sum_{j=1}^{m-1}d_j^2-\sum_{i=1}^m n_i^2
	=\sum_{j=1}^{m-1}(d_j-n_{j+1})(d_j+n_{j+1})-n_1^2
	\geq n_1^2>0,
	\]
	a contradiction. Thus \(\B\), and hence
	\(c\A((C_m)_{\Sgain})\), is not free.
\end{proof}

\subsection{Chordality and Terao's conjecture}

\begin{proof}[Proof of Theorem~\ref{thm:free-implies-chordal}]
	Suppose that \(G\) is not chordal. Then \(G\) contains an induced cycle
	\(C\) of length at least four. Let \(X\) be the flat of
	\(c\A(G_{\Sgain})\) defined by
	\[
	y=0,
	\qquad
	x_u=x_v\quad \text{for all }u,v\in V(C).
	\]
	This is indeed a flat: if one chooses a gain \(a_e\in S_e\) on each edge
	\(e\) of \(C\), where \(S_e\) denotes the gain set on that edge, then
	\(X\) is the intersection of \(H_\infty\) with the hyperplanes
	\(cH_e^{a_e}\), \(e\in E(C)\).
	A cone hyperplane \(cH_{ij}^a\) contains \(X\) precisely when both
	endpoints \(i,j\) lie in \(V(C)\). Since \(C\) is induced, these are
	exactly the edges of \(C\). Therefore the localization
	\((c\A(G_{\Sgain}))_X\) is, up to empty product factors, the cone
	\(c\A(C_{\Sgain|C})\), where \(\Sgain|C\) denotes the restriction of
	\(\Sgain\) to the edges of \(C\).
	Localizations of free arrangements are free \cite[Chapter~4]{OrlikTerao1992}.
	Thus freeness of
	\(c\A(G_{\Sgain})\) would imply freeness of \(c\A(C_{\Sgain|C})\), in
	contradiction to Lemma~\ref{lem:cycle-not-free}. Hence \(G\) has
	no induced cycle of length at least four, and so \(G\) is chordal.
\end{proof}

\begin{remark*}
	Assume in this remark that \(\operatorname{char}\K=0\).
	Theorem~\ref{thm:free-implies-chordal} cannot be strengthened from
	``\(G\) is chordal'' to ``\(G_{\Sgain}\) has an
	\(\Sgain\)-admissible ordering''. Indeed, the cones
	over the classical Shi and Catalan arrangements of type \(A_{n-1}\) are
	free \cite{AbeTerao2011}. Their underlying graph is \(K_n\), hence
	chordal. However, for \(n\ge3\), neither arrangement admits an
	\(\Sgain\)-admissible ordering; this is verified in
	Example~\ref{ex:shi-catalan-linial} below. Thus freeness rules out
	chordless cycles, but it does not detect the full gain-admissibility
	condition.
\end{remark*}

We conclude by relating Theorem~\ref{thm:free-implies-chordal} to Terao's
conjecture; in this discussion assume that \(\operatorname{char}\K=0\).
The conjecture asks whether freeness of a central arrangement over a fixed
field is determined by its intersection lattice; see
\cite[Chapter~4]{OrlikTerao1992}. Theorem~\ref{thm:free-implies-chordal} is not a freeness criterion for the cones
\(c\A(G_{\Sgain})\). It gives a necessary obstruction detected by
localization: if \(G\) contains an induced chordless cycle, then the cone has
a non-free localization over \(H_\infty\).

Let \(\B=c\A(G_{\Sgain})\). The Ziegler restriction of \(\B\) to
\(H_\infty\) is the graphic multiarrangement \cite{Ziegler1989}
\[
(\B^{H_\infty},m)=(\A(G),m),
\qquad
m(H_{ij}^{0})=|S_{ij}|.
\]
Thus, once \(H_\infty\) is distinguished, the pair
\((L(\B),H_\infty)\) records not only the graphic arrangement at infinity,
but also the multiplicities determined by the gain sets. Hence
the freeness problem for these cones is not reduced to the ordinary graphic
arrangement: the affine incidences and the multiplicities at infinity both
enter the lattice.

When all gain sets are \(\{0\}\), the cone is the product of the ordinary
graphic arrangement with a one-dimensional Boolean arrangement. In that
case freeness is equivalent to chordality
\cite{Stanley1972,EdelmanReiner1994}. Theorem~\ref{thm:free-implies-chordal} extends the necessary direction of this
criterion to arbitrary finite gain sets: freeness of \(c\A(G_{\Sgain})\)
still forces \(G\) to be chordal.

\section{Characteristic Polynomial and Examples}

\subsection{Characteristic polynomial}

\begin{corollary}\label{cor:characteristic-polynomial}
	Let \(G\) have \(c\) connected components and let
	\(G_1,\ldots,G_s\) be its blocks containing at least one edge. Assume that
	\(G_{\Sgain}\) is blockwise admissible,
	and choose an admissible ordering
	\(v_{\nu,1},\ldots,v_{\nu,m_\nu}\) for every
	\((G_\nu)_{\Sgain_\nu}\). Put
	\[
	d_{\nu,k}
	=
	\sum_{i>k}|S_{v_{\nu,k}v_{\nu,i}}|
	\qquad
	(1\le k<m_\nu).
	\]
	Then
	\[
	\chi(\A(G_{\Sgain}),t)
	=
	t^c
	\prod_{\nu=1}^{s}
	\prod_{k=1}^{m_\nu-1}(t-d_{\nu,k}).
	\]
\end{corollary}

\begin{proof}
	For every block, the inductive construction in Proposition~\ref{prop:order-to-ss-infty} gives a modular chain through its hyperplane
	at infinity. At the step that removes \(v_{\nu,k}\), the hyperplanes not
	containing the new coatom are precisely
	\(cH_{v_{\nu,k}v_{\nu,i}}^a\), where \(i>k\) and
	\(a\in S_{v_{\nu,k}v_{\nu,i}}\). Thus the nice partition induced by this
	chain has part sizes \(
	1,d_{\nu,1},\ldots,d_{\nu,m_\nu-1}\).
	Iterating Lemma~\ref{lem:cone-product}\ref{cone-product:gluing} identifies
	the copies of \(H_\infty\) into one singleton part and leaves all other
	part sizes unchanged. The center of \(c\A(G_{\Sgain})\) has dimension
	\(c\). The factorization associated with this nice partition
	\cite{Terao1992} therefore gives
	\[
	\chi(c\A(G_{\Sgain}),t)
	=
	t^c(t-1)
	\prod_{\nu=1}^{s}
	\prod_{k=1}^{m_\nu-1}(t-d_{\nu,k}).
	\]
	The standard identity \(\chi(c\A,t)=(t-1)\chi(\A,t)\) yields the stated
	formula.
\end{proof}

\begin{remark*}
	Assume that \(G\) is a block with at least two vertices, and let
	\(v_1,\ldots,v_n\) be an admissible ordering for \(G_{\Sgain}\).
	Choose any finite set \(T\subseteq\K\) containing
	\(S_{v_{n-1}v_n}\), and replace \(S_{v_{n-1}v_n}\) by \(T\), together
	with the corresponding replacement \(S_{v_nv_{n-1}}=-T\). The same
	ordering is still admissible, because \(S_{v_{n-1}v_n}\) occurs only on
	the right-hand side of the defining inclusions. Thus, by Theorem~\ref{thm:main}, adding the hyperplanes
	\[
	H_{v_{n-1}v_n}^{a}\qquad (a\in T\setminus S_{v_{n-1}v_n})
	\]
	preserves supersolvability of the cone.
\end{remark*}

\subsection{Examples}

\begin{example}[Graphic arrangements]
	If \(S_{ij}=\{0\}\) for every edge \(ij\), then
	\(\A(G_{\Sgain})\) is the ordinary graphic arrangement \(\A(G)\).
	Here admissible orderings are precisely perfect elimination orderings.
	Hence Corollary~\ref{cor:blockwise-general} recovers Stanley's
	characterization of supersolvable graphic arrangements: \(\A(G)\) is
	supersolvable if and only if \(G\) is chordal.
\end{example}

\begin{example}[A four-cycle and a chord]\label{ex:four-cycle-chord}
	Let \(C_4\) be the cycle on vertices \(1,2,3,4\). For any nonempty
	gain sets on its edges, Lemma~\ref{lem:cycle-not-free} gives that
	\(c\A((C_4)_{\Sgain})\) is not free.
	
	Now add the chord \(13\), and call the resulting graph \(G\). Then
	\(G\) is chordal; for instance, \(2,4,1,3\) is a perfect elimination
	ordering. The same ordering is \(\Sgain\)-admissible precisely when
	\[
	S_{23}-S_{21}\subseteq S_{13}
	\qquad\text{and}\qquad
	S_{43}-S_{41}\subseteq S_{13}.
	\]
	Thus this ordering is admissible exactly under these two gain conditions;
	whenever they hold, Theorem~\ref{thm:main} makes the cone supersolvable.
	For ordinary
	gains \(S_{ij}=\{0\}\), both inclusions hold. In contrast, if
	\(S_{ij}=\{0,1\}\) for the four cycle edges written with \(i<j\), and
	\(S_{13}=\{0\}\) for the chord, then the required inclusion fails for
	both simplicial vertices \(2\) and \(4\). These are the only simplicial
	vertices of \(G\), so there is no \(\Sgain\)-admissible ordering. Thus
	chordality alone is not enough for supersolvability in this setting.
\end{example}

\begin{example}[Nested \(N\)-Ish arrangements]
	Over a field of characteristic zero, Abe, Suyama, and Tsujie proved
	that the cone over an \(N\)-Ish
	arrangement is supersolvable precisely when the sets
	\(N_2,\ldots,N_n\) form a nest \cite{AbeSuyamaTsujie2017}. In our
	notation, choose finite sets \(N_2,\ldots,N_n\subseteq\K\) and set
	\[
	S_{1j}=N_j\quad (2\le j\le n),
	\qquad
	S_{ij}=\{0\}\quad (2\le i<j\le n).
	\]
	Let \(u_1,\ldots,u_{n-1}\) be a permutation of \(2,\ldots,n\). Then
	\(u_1,\ldots,u_{n-1},1\) is an admissible ordering if and only if \(
	N_{u_1}\subseteq N_{u_2}\subseteq\cdots\subseteq N_{u_{n-1}}\).
	Indeed, the only nontrivial inclusions are
	\[
	S_{u_r u_s}-S_{u_r1}
	=
	\{0\}-(-N_{u_r})
	=
	N_{u_r}
	\subseteq
	N_{u_s}
	=
	S_{1u_s}
	\qquad (r<s).
	\]
	Thus Corollary~\ref{cor:blockwise-general} gives supersolvability
	for every nested family \(N_2,\ldots,N_n\).
\end{example}

\begin{example}[Shi, Catalan, and Linial arrangements]\label{ex:shi-catalan-linial}
	Assume that \(\operatorname{char}\K=0\), and let \(G=K_n\), with
	\(n\ge3\). The classical Shi, Catalan, and Linial
	arrangements are obtained, respectively, by taking
	\[
	S_{ij}=\{0,1\},\qquad
	S_{ij}=\{-1,0,1\},\qquad
	S_{ij}=\{1\}
	\]
	for \(i<j\); see \cite{PostnikovStanley2000}. None of the corresponding
	cones is supersolvable. Let \(v_k\) be the first vertex in any
	ordering, and choose two later vertices \(v_i,v_j\).
	\begin{itemize}[leftmargin=2em]
		\item For the Shi arrangement, the two sets \(S_{v_kv_i}\) and
		\(S_{v_kv_j}\) are each either \(\{0,1\}\) or \(\{0,-1\}\). If the
		two signs agree, their difference contains both \(1\) and \(-1\),
		while \(S_{v_iv_j}\) contains only one nonzero sign. If the signs
		differ, their difference contains \(2\) or \(-2\). In either case the
		admissibility inclusion fails.
		\item For the Catalan arrangement,
		\[
		S_{v_kv_j}-S_{v_kv_i}\supseteq \{-2,-1,0,1,2\},
		\]
		which is not contained in \(S_{v_iv_j}\).
		\item For the Linial arrangement, each set \(S_{ij}\) is either
		\(\{1\}\) or \(\{-1\}\), while
		\(S_{v_kv_j}-S_{v_kv_i}\) is either \(\{0\}\), \(\{2\}\), or
		\(\{-2\}\). None of these sets is contained in \(S_{v_iv_j}\).
	\end{itemize}
	In each case admissibility fails for every ordering, so Theorem~\ref{thm:main} shows that the cone is not supersolvable and the
	affine arrangement has no nice partition.
\end{example}

\section*{Declaration on the Use of AI}

During the preparation of this manuscript, the authors used AI-assisted
tools for language editing, \LaTeX{} organization, consistency and
computational checks, and the generation and exploration of possible
mathematical arguments. In particular, AI-generated suggestions played a
substantial role in developing the argument proving that freeness of
\(c\A(G_{\Sgain})\) forces \(G\) to be chordal.
The authors independently verified every step of this argument and reviewed,
revised, and finalized all AI-assisted output. The authors take full
responsibility for the correctness and content of the manuscript.

\end{document}